\documentclass[12pt, reqno]{amsart}
\usepackage{amssymb,dsfont}
\usepackage{eucal}
\usepackage{amsmath}
\usepackage{amscd}
\usepackage[dvips]{color}
\usepackage{multicol}
\usepackage[all]{xy}           
\usepackage{graphicx}
\usepackage{color}
\usepackage{colordvi}
\usepackage{xspace}
\usepackage{txfonts}
\usepackage{lscape}
\usepackage{tikz} 
\usepackage{bm}

\numberwithin{equation}{section}
\usepackage[shortlabels]{enumitem}
\usepackage{ifpdf}
\ifpdf
 \usepackage[colorlinks,final,backref=false,hyperindex]{hyperref} 
\else
\fi
\usepackage{tikz}
\usepackage[active]{srcltx}

\makeatletter

\newtheorem{theorem}{Theorem}[section]
\newtheorem{definition}[theorem]{Definition}
\newtheorem{lemma}[theorem]{Lemma}
\newtheorem{proposition}[theorem]{Proposition}
\newtheorem{remark}[theorem]{Remark}
\newtheorem{corollary}[theorem]{Corollary}

\numberwithin{equation}{section}

\def\lmd{\lambda}
\def\al{\alpha}
\def\ep{\epsilon}
\def\be{\beta}
\def\dt{\delta}
\def\h{\mathfrak{h}}
\def\g{\mathfrak{g}}
\def\vf{\varphi}
\def\vvf{v_{\vf}}
\def\jc{J^C}
\def\gjc{{\g_{\jc}}}
\def\gic{{\g_{I^C}}}
\def\hi{\h_I}
\def\gi{\g_I}
\def\gizero{\gi^{(0)}}
\def\gim{\gi^{(m)}}
\def\mgivf{M_{\gi,\vf}}

\def\({\left(}
\def\){\right)}
\def\p{\partial}

\def\U{\mathcal{U}}
\def\V{\mathcal{V}}

\newcommand{\C}{\mathbb C}

\newcommand{\R}{\mathbb{R}}
\newcommand{\N}{\mathbb{N}}
\newcommand{\Z}{\mathbb{Z}}
\newcommand{\CN}{\mathcal{N}}

\newcommand{\spanc}[1]{\mathrm{Span}_{\C}\left\{#1\right\}}
\newcommand{\un}[1]{\underline{#1}}
\newcommand{\ov}[1]{\overline{#1}}

\def\Ind{\mathrm{Ind}}
\def\supp{\mathrm{Supp}}

\newcommand{\elei}[2]{I_{#1}^{(#2)}}
\newcommand{\vfi}[2]{\vf\left(I_{#1}^{(#2)}\right)}
\newcommand{\vfl}[1]{\vf(L_{#1})}

\begin{document}
\title[Non-weight module over gap-$p$ Virasoro algebras]{Non-weight modules over gap-$p$ Virasoro algebras}
  \author{Chengkang Xu}
  \address{C. Xu: School of Mathematical Sciences, Shangrao Normal University, Shangrao, Jiangxi, P. R. China}
  \email{xiaoxiongxu@126.com}
  \author{Fulin Chen}
  \address{F. Chen: School of Mathematical Sciences, Xiamen University, Xiamen, Fujian, P. R. China}
  \email{chenf@xmu.edu.cn}
    \author{Shaobin Tan}
  \address{S. Tan: School of Mathematical Sciences, Xiamen University, Xiamen, Fujian, P. R. China}
  \email{tans@xmu.edu.cn}

\date{}
\maketitle

\begin{abstract}
In this paper, we study non-weight modules over gap-$p$ Virasoro algebras,
including Whittaker modules, $\mathcal{U}(\mathbb{C} L_0)$-free modules and their tensor products.
We establish necessary and sufficient conditions for universal Whittaker modules to be irreducible and study the structure of irreducible Whittaker modules.
The $\mathcal{U}(\mathbb{C} L_0)$-free modules of rank 1 are classified and the irreducibility of such modules are determined.
Moreover, the irreducibility of tensor products of $\mathcal{U}(\mathbb{C} L_0)$-free modules of rank 1 and irreducible restricted modules is also determined.
\end{abstract}

\section{Introduction}

In the past decades, representation theory of those Lie algebras that are closely related to the Virasoro algebra has been extensively studied.
Particularly, the theory of weight modules has been well-developed,
such as the classification of irreducible Harish-Chandra modules over the twisted Heisenberg-Virasoro algebra \cite{LZ1},
the Schr{\"o}dinger-Virasoro algebras \cite{LS}, the mirror Heisenberg-Virasoro algebra \cite{LPXZ},
certain deformed Heisenberg-Virasoro algebras \cite{Liu}, the gap-$p$ Virasoro algebras \cite{Xu},
certain Lie algebras of block-type \cite{GGS,SXX,WT} and etc.
There are also some interesting studies on other kinds of weight modules, see \cite{BBFK,CFM} for examples.

In recent years, the theory of
non-weight modules over infinite dimensional Lie algebras has also attracted more and more attention.
Among them, Whittaker modules (or more generally, non-weight restricted modules) and $\U(\h)$-free modules are most well-known.  For a Lie algebra with a triangular decomposition,  its Whittaker module is a module on which the elements in the positive part (or its certain subalgebra) act as scalars.
This class of modules was first constructed for $\mathfrak{sl}_2$ in \cite{AP},
and then systematically studied for all semisimple Lie algebras in \cite{Kon}.
 Whittaker modules for infinite dimensional Lie algebras have been  intensively studied as well,
including the Virasoro algebra \cite{OW,LGZ}, the generalized Weyl algebras \cite{BO},
the affine Lie algebras \cite{Chr,ALZ}, the twisted Heisenberg-Virasoro algebra \cite{LZ3},
the mirror Heisenberg-Virasoro algebra \cite{GMZ},
the Schr{\"o}dinger-Witt algebra \cite{ZTL}, and the rank two Heisenberg-Virasoro algebra \cite{TWX}.

Restricted module is a generalization of highest weight module and Whittaker module,
which is closely connected to (weak) modules for certain vertex algebras.
To classify irreducible restricted modules is a core topic in the representation theory of Lie algebras and related vertex algebras.
The problem was studied for several Lie (super)algebras, such as the Virasoro algebra \cite{MZ},
Neveu-Schwarz algebras \cite{LPX}.
the twisted and mirror Heisenberg-Virasoro algebras \cite{CG1,Gao,LPXZ,TYZ},
gap-$p$ Virasoro algebras of level zero \cite{GX}, and so on.

For a Lie algebra  $\mathcal L$ equipped with a Cartan subalgebra $\h$, its
$\U(\h)$-free module is defined to be a module free of a fixed rank when viewed as a $\U(\h)$-module.
Such modules were first introduced in \cite{Nil1} for the simple Lie algebra $\mathfrak{sl}_{n+1}$.
Since then, $\U(\h)$-free modules have been widely studied for various Lie algebras,
including finite dimensional simple Lie algebras \cite{Nil2}, Kac-Moody algebras \cite{CTZ},
the Virasoro algebra \cite{LZ2}, the Witt algebras $W_n$ and $W_n^+$ \cite{TZ},
the twisted Heisenberg-Virasoro algebra and the $W$-algebra $W(2,2)$ \cite{CG2},
the mirror Heisenberg-Virasoro algebra\cite{GMZ}, and so on.

The gap-$p$ Virasoro algebra $\g$ was first introduced in \cite{Xu}, where irreducible Harish-Chandra modules were classified.
It is closely related to the algebra of derivations over a rational quantum torus as well as the twisted Heisenberg-Virasoro algebra.
Particularly, when $p=2$, $\g$ is just the mirror Heisenberg-Virasoro algebra.
In this paper, we study non-weight modules over the gap-$p$ Virasoro algebra $\g$,
including Whittaker modules, $\U(\C L_0)$-free modules and their tensor products.
The main goals of this paper are to determine the irreducibility of universal Whittaker $\g$-modules;
classify the $\U(\C L_0)$-free $\g$-modules of rank $1$;
and determine the irreducibility of the tensor products of $\U(\C L_0)$-free $\g$-modules and
irreducible restricted $\g$-modules.

The structure of this paper is as follows.
In Section 2, we recall some notions and results about Whittaker modules, $\U(\C L_0)$-free modules, as well as
restricted modules.
In Section 3, we first give a free field construction of certain Whittaker $\g$-modules. Using this, we construct certain tensor product $\g$-modules.
Then we characterize all irreducible Whittaker $\g$-modules by determining the irreducibility of universal Whittaker modules of level zero over a subalgebra $\gi$ of $\g$.
In Section 4, we classify all $\g$-modules that are free of rank 1 over $\U(\C L_0)$.
In Section 5, we study the tensor products of the $\U(\C L_0)$-free modules and irreducible restricted modules over $\g$.
Finally, in the appendix, we give a vertex algebra interpretation of the free field construction of Whittaker $\g$-modules, which leads to a realization of $\g$ as a twisted vertex Lie algebra.

Throughout this paper, we denote by $\Z$, $\N$, $\Z_+$, $\mathbb R,\,\C$, and $\C^\times$ the set of integers, non-negative integers, positive integers, real numbers, complex numbers, and nonzero complex numbers,  respectively.
All vector spaces and algebras are  over $\C$.
For a Lie algebra $\mathcal L$, we denote by $\U(\mathcal L)$ the universal enveloping algebra of $\mathcal L$.
For a subalgebra $\mathcal K$  of $\mathcal L$ and  a $\mathcal K$-module $V$,
we  denote by  \[\text{Ind}_{\mathcal K}^{\mathcal L}V=\U(\mathcal L)\otimes_{\U(\mathcal K)}V\] the $\mathcal L$-module induced from $V$.
For any $n\in\R$, we denote by $[n]$ the largest integer that is less than or equal to $n$.
Finally, we fix a positive integer $p>1$ for this paper.

\section{Preliminaries}
 In this section, we introduce some notations related to the gap-$p$ Virasoro algebra, and recall some results on the  Whittaker modules and $\U(\C L_0)$-free modules over the Virasoro algebra.

\begin{definition}
The \textbf{gap-$p$ Virasoro algebra} $\g$ is a complex Lie algebra with a basis
$$\left\{L_n,\elei ni, C_j\ \mid\ \, n\in\Z, 1\leq i\leq p-1, 0\leq j\leq [\frac p2] \right\}$$
satisfying the following Lie brackets
\begin{align}
  &[L_m,L_n]=(m-n)L_{m+n}+\frac1{12}(m^3-m)C_0\dt_{m+n,0},              \label{eq2.1}\\
  &[\elei mi,\elei nj]=\left(m+\frac ip\right)\dt_{i+j,p}\dt_{m+n+1,0} C_i,\ \ \ \ \
                                               [C_k, \g]=0,             \label{eq2.2}\\
  &[L_m,\elei ni]=-\left(n+\frac ip\right)\elei{m+n}i,                   \label{eq2.3}
\end{align}
where $m,n\in\Z$, $1\leq i,j\leq p-1$, and $0\leq k\leq [\frac p2]$.
\end{definition}

Here and below, for $[\frac p2]<i\leq p-1$, we also set $C_i=C_{p-i}$ for narrative convenience.

Note that the Lie algebra $\g$ contains a subalgebra
\[
\V=\mathrm{Span}_\C\{L_n,C_0 \mid n\in\Z\}
\]
 which is isomorphic to the Virasoro algebra, and contains an ideal
\[
\h=\mathrm{Span}_\C\{\elei ni,C_i \mid  n\in\Z, 1\leq i\leq p-1\}
\]
which is isomorphic to a direct sum of Heisenberg algebras.
Obviously $\g$ is the semi-direct product of $\V$ and $\h$.

Let $I$ be a subset  of $\{1,2,\dots,p-1\}$.
We denote by $|I|$ the cardinal number of $I$, and set
\[I^C=\{1,2,\dots,p-1\}\setminus I,\qquad
I_0=I\cap \left\{1,2,\dots,[\frac{p}{2}]\right\}.
\]
Associated to $I$, we have the following two subspaces of $\g$:
\begin{eqnarray}
 \g_I&=&\spanc{L_n,\elei ni, C_0, C_i\ \, \vline\ \, n\in\Z,  i\in I}, \label{defg_I}\\
 \h_I&=&\spanc{\elei ni, C_i\ \mid\  n\in\Z,  i\in I}.\label{defh_I}
\end{eqnarray}
Clearly, $\g_I$ is a subalgebra of $\g$, $\h_I$ is an ideal of $\g$, and
\[
\g_I=\V\oplus \h_I.
\]
Note that $\g_I=\V$, $\h_I=0$ when $I=\emptyset$, and $\g_I=\g$, $\h_I=\h$ when $I=\{1,2,\dots,p-1\}$.

In this paper we mainly concern about the case that $I$ is {\bf symmetrical}, meaning that $p-i\in I$ whenever $i\in I$.
In this case, the complement set $I^C$ is also symmetrical, and
\begin{eqnarray}\label{eq:dechI}
 \h_I=\oplus_{i\in I_0}\h_{\{i,p-i\}}\qquad \left(\h_{\{i,p-i\}}
 =\spanc{I_n^{(i)},I_n^{(p-i)},C_i\mid n\in \Z}\right)
\end{eqnarray}
is a direct sum of Heisenberg algebras.
Additionally, we have that
\[
\g=\g_I\oplus \h_{I^C}\quad\text{and}\quad
\h=\h_I\oplus \h_{I^C}.
\]

For a $\Z$-graded Lie algebra $\mathcal{L}$, we denote by $\mathcal{L}_m$ ($m\in \Z$) the $m$-th component of it such that
$\mathcal{L}=\oplus_{m\in \Z}\mathcal{L}_m$.

\begin{definition}\label{Restricted}
A module $V$ over a $\Z$-graded Lie algebra $\mathcal L$ is called \textbf{restricted}
if for any $v\in V$, there exists $N\in\N$ such that $\mathcal{L}_m v=0$ for all $m\geq N.$
\end{definition}

The gap-$p$ Virasoro algebra $\g$ admits a $\Z$-graded structure with
\begin{eqnarray*}
  \g_0&=&\spanc{L_0, C_0, \elei 0i, C_i\ \, \vline\ \, 1\leq i\leq p-1},\ \text{and}\\
  \g_n&=&\spanc{L_n, \elei ni\ \, \vline\ \, 1\leq i\leq p-1}\quad\text{for}\  n\neq 0.
\end{eqnarray*}
Note that for every subset $I$ of $\{1,2,\dots,p-1\}$,
$\g_I$ and $\h_I$ are both $\Z$-graded subalgebras of $\g$.
In particular, $\V,\h,\gi,\hi$ are all $\Z$-graded Lie algebras.

\begin{definition}
 Let $I$ be a symmetrical subset of $\{1,2,\dots,p-1\}$,
 and let $\un l=(l_i)_{i\in I}$ be an $|I|$-tuple in $ \C^{|I|}$.
 We say that a module $V$ over $\g_I$ or $\h_I$ is of \textbf{level} $\un l$ if  for every $i\in I$,
 the central element $C_i$ acts as the scalar $l_i$ on $V$.
 Additionally,  we say that the $|I|$-tuple $\un l$ is \textbf{generic} if $l_i\neq 0$ for all $i\in I$.
\end{definition}
Since $I$ is symmetrical, when we say a level $\un l$ of a module over $\g_I$ or $\h_I$, the $|I|$-tuple  $\un l$ is always assumed to be  symmetrical in the sense  that $l_{p-i}=l_i$ for all $i\in I$.

\begin{definition}
  Let $c\in \C$. A $\g_I$-module $V$ is said to be \textbf{of central charge} $c$ if $C_0$ acts as $c$ on $V$.
\end{definition}

\begin{definition}
Let $\mathcal{L}$ be a $\Z$-graded Lie algebra and let $\mathcal{K}$ be a graded subalgebra of $\mathcal{L}$ such that $\mathcal{K}_{-m}=0$ for all $m\in \N$. A \textbf{Whittaker function} $\phi$ on $\mathcal{K}$ is defined as a Lie algebra homomorphism $\phi:\mathcal{K}\rightarrow\C$.
An $\mathcal{L}$-module is called a \textbf{Whittaker module} of type $\phi$ if it can be generated by a \textbf{Whittaker vector} of type $\phi$, namely,  a vector $v$ satisfying
 \[
   x\cdot v=\phi(x)v\qquad \text{for all}\ x\in \mathcal{K}.
 \]
\end{definition}

Every Whittaker function $\phi$ on $\mathcal{K}$ defines a one-dimensional $\mathcal{K}$-module
$\C v_\phi $ such that
\[
 x\cdot v_\phi=\phi(x) v_\phi\qquad \text{for all}\ x\in \mathcal{K}.
\]
The resulting induced $\mathcal{L}$-module
\[
M_{\mathcal{L},\phi}=\Ind_{\mathcal{K}}^{\mathcal{L}} \C v_\phi
 \quad (=\U(\mathcal{L})\otimes_{\U(\mathcal{K})} \C v_\phi)
\]
is often referred as the \textbf{universal Whittaker module} of type $\phi$,
and $v_\phi$ a \textbf{universal Whittaker vector} of type $\phi$. Note that for every Whittaker $\mathcal L$-module $V$ of type $\phi$ with Whittaker vector $v$, there is an $\mathcal L$-module epimorphism
$\pi: M_{\mathcal{L},\phi}\longrightarrow V$ such that $\pi(v_\phi)=v$.

In what follows, we recall some facts on Whittaker modules and $\U(\C L_0)$-free modules   over the Virasoro algebra.

\begin{proposition}\cite[Theorem 7]{LGZ}\label{prop2.5}
Let $m\in\Z_+$  and let $\phi$ be a Whittaker function on the subalgebra $\V_{\geq m}=\oplus_{k\geq m}\C L_k$ of $\V$. Then the universal Whittaker $\V$-module $M_{\V,\phi}$ is irreducible
 if and only if $\left(\phi(L_{2m-1}),\phi(L_{2m})\right)\neq (0,0)$.
\end{proposition}

For every $\lmd\in\C^\times$ and $a\in\C$, the polynomial ring $\C[t]$ carries
a $\V$-module structure defined by
$$C_0 f(t)=0,\ \ \ \ L_nf(t)=\lmd^n(t+an)f(t+n)\ \ \ \textit{ for all }\ n\in\Z,\ f(t)\in\C[t].$$
The resulting $\V$-module is denoted by $\Omega(\lmd,a)$.
The following results are proved in \cite{LZ2} and \cite[Theorem 3]{TZ}.

\begin{proposition}\label{prop2.6}
(1) The $\V$-module $\Omega(\lmd,0)$ has a unique irreducible submodule $t\Omega(\lmd,0)$.\\
\notag (2) The $\V$-module $\Omega(\lmd,a)$ is irreducible if and only if $a\neq0$.\\
\notag(3) Let $V$ be a $\V$-module which is free of rank 1 when viewed as a $\U(\C L_0)$-module.
 Then $V$ is isomorphic to $\Omega(\lmd,a)$ for some $\lmd\in\C^\times$ and $a\in\C$.
\end{proposition}

We also record the following elementary results as a lemma for later use (cf. \cite[Lemma 6, Theorem 7]{LZ3}).
\begin{lemma}\label{lem2.7}
 Let $\mathcal L$ be a Lie algebra of countable dimension.\\
 (1) Let $V$ be an irreducible $\mathcal L$-module.
 For any $n\in\Z_+$ and any linearly independent subset $\{v_1,v_2,\dots,v_n\}\subset V$,
 and any subset $\{v'_1,v'_2,\dots,v'_n\}\subset V$, there exists $u\in\U(\mathcal L)$ such that
 $$uv_i=v'_i\quad\text{ for all } 1\le i\le n.$$
 (2) Let $V_1,V_2$ be $\mathcal L$-modules.
 If for any finite subset $S$ of $V_2$, $V_1$ is an irreducible module over the subalgebra
 $\mathrm{ann}_L(S)=\{x\in\mathcal L\mid xS=0\}$ of $\mathcal L$,
 then any $\mathcal L$-submodule of $V_1\otimes V_2$ is of the form $V_1\otimes V'_2$,
 where $V'_2$ is a submodule of $V_2$.
\end{lemma}

\section{Whittaker $\g$-modules}
Throughout this section, $m$ is a fixed positive integer. We have the following graded subalgebra
$$\g^{(m)}=\spanc{L_{m+k},\elei ki,C_0,C_i\ \mid\ k\in\N,1\leq i\leq p-1}$$
of $\g$.
Associated to the subalgebra $\g^{(m)}$, we  have the notions of a Whittaker function $\varphi$,
a Whittaker $\g$-module of type $\varphi$, and the  universal Whittaker $\g$-module
$M_{\g,\varphi}$ of type $\varphi$ with a universal Whittaker vector $\vvf$.
Note that for any Whittaker function $\varphi$ on $\g^{(m)}$, one has
\begin{equation}\label{eq3.1}
\vf(L_{2m+j})=\vf\(\elei {m+j-1}i\)=0\ \ \ \ \text{for any }j\in\Z_+, 1\le i\le p-1.
\end{equation}

In this section we give a free field realization of certain irreducible Whittaker $\g$-modules, and
determine the irreducibility of universal Whittaker $\g$-modules. To this end, we need to work in a more general setting.
Let $I$ be a fixed nonempty symmetrical subset  of $\{1,2,\dots,p-1\}$ throughout this section.
We have the subalgebra
\[
\g_I^{(m)}=\g^{(m)}\cap \g_I
\]
of $\g_I$, and the notion of a Whittaker function $\phi$ on $\g_I^{(m)}$.
We will consider Whittaker $\g_I$-modules of type $\phi$, including the universal one $M_{\g_I,\phi}$.

\subsection{Free field realizations of irreducible Whittaker $\g$-modules}
We start by constructing restricted $\g$-modules from restricted $\h_I$-modules of generic levels.

\begin{proposition}\label{proprealize}
 Let $V$ be a restricted $\hi$-module of generic level $\un l=(l_j)_{j\in I}\in \C^{|I|}$.
 Then there exists a restricted $\g$-module structure on $V$ given by
 \begin{align}
  \elei ni &\mapsto \begin{cases} \elei ni\ &\text{if}\ i\in I,\\
                                        0\ &\text{if}\ i\notin I,\end{cases}\ \ \ \ \quad
  C_i\mapsto \begin{cases} l_i\ &\text{if}\ i\in I,\\
                            0 \ &\text{if}\ i\notin I,\end{cases},\quad C_0\mapsto |I|,\label{eq:Iaction}\\
    L_n   &\mapsto \sum_{j\in I}\frac{1}{2l_j}\left(\sum_{k\in \Z} :
                    \elei kj\elei{n-k-1}{p-j}:\right)
                   +\dt_{n,0}\sum_{j\in I}\frac{j(p-j)}{4p^2},       \label{eq:Lnaction}
 \end{align}
 where $n\in\Z, 1\leq i\leq p-1$ and the normal ordered product is defined by
\begin{equation*}
   :\elei ki\elei l{p-i}:=\begin{cases}
      \elei ki\elei l{p-i}\quad&\text{ if }k<l,\\
      \elei l{p-i}\elei ki \quad&\text{ if }k\ge l.
   \end{cases}
\end{equation*}
\end{proposition}

\begin{proof}
 It is clear that the operators given in \eqref{eq:Iaction} and \eqref{eq:Lnaction} satisfy the relation \eqref{eq2.2}.
 In what follows, we prove that they also satisfy the relations \eqref{eq2.1} and \eqref{eq2.3}.

 Using the ``cutoff" procedure,  for any $n\in\Z$ we set
 $$L_n(\ep)=\sum_{j\in I}\frac{1}{2l_j}\left(\sum_{k\in \Z} :\elei kj\elei{n-k-1}{p-j}:\psi(k\ep)\right),$$
 where the function $\psi:\R\longrightarrow\R$ is defined by
 $$\psi(x)=1\ \text{ if }|x|\leq 1,\qquad \psi(x)=0\ \text{ if }|x|> 1.$$
 Notice that $L_n(\ep)\rightarrow L_n$ as $\ep\rightarrow0$, and $L_n(\ep)$ differs from
 $$\sum_{j\in I}\frac{1}{2l_j}\left(\sum_{k\in \Z} \elei kj\elei{n-k-1}{p-j}\psi(k\ep)\right)$$
 only by a scalar. For any $i\in I,s,n\in\Z$, we have
 \begin{align*}
  &[\elei ni,L_s(\ep)]=\sum_{j\in I}\frac{1}{2l_j}\left(\sum_{k\in \Z}
    [\elei ni,\elei kj\elei{s-k-1}{p-j}]\psi(k\ep)\right)\\
  &=\sum_{j\in I}\frac{1}{2l_j}(n+\frac ip)\left(\sum_{k\in \Z}
       \dt_{n+k+1,0}\dt_{i+j,p}\elei{s-k-1}{p-j}l_i\psi(k\ep)+
         \dt_{n+s,k}\dt_{i,j}\elei{k}{j}l_i\psi(k\ep)  \right)\\
  &=\frac12(n+\frac ip)\elei {s+n}i\Big(\psi\big(\ep(n+1)\big)+\psi\big(\ep(s+n)\big)\Big).
 \end{align*}
 By taking the limit $\ep\rightarrow0$, it gives \eqref{eq2.3}.

 For any $s,n\in\Z$, we have
 \begin{align*}
  &[L_s(\ep),L_n]=\sum_{j\in I}\frac{1}{2l_j}\left(\sum_{k\in \Z}
                 [\elei kj\elei{s-k-1}{p-j},L_n]\psi(k\ep)\right)\\
  =&\sum_{j\in I}\frac{1}{2l_j}\left(\sum_{k\in \Z}(k+\frac jp)
     \elei {n+k}j\elei{s-k-1}{p-j}\psi(k\ep)\right)+
     \sum_{j\in I}\frac{1}{2l_j}\left(\sum_{k\in \Z}(s-k-\frac jp)
      \elei {k}j\elei{s+n-k-1}{p-j}\psi(k\ep)\right)\\
  =&\sum_{j\in I}\frac{1}{2l_j}\left(\sum_{k\in \Z}(k+\frac jp)
                                :\elei {n+k}j\elei{s-k-1}{p-j}:\psi(k\ep)
                                +\sum_{k\geq\frac{s-n-1}2}(k+\frac jp)
          [\elei {n+k}j,\elei{s-k-1}{p-j}]\psi(k\ep)\right)\\
       &+\sum_{j\in I}\frac{1}{2l_j}\left(\sum_{k\in \Z}(s-k-\frac jp)
                                :\elei {k}j\elei{s+n-k-1}{p-j}:\psi(k\ep)
                                +\sum_{k\geq\frac{s+n-1}2}(s-k-\frac jp)
        [\elei {k}j,\elei{s+n-k-1}{p-j}]\psi(k\ep)\right)\\
  =&\sum_{j\in I}\frac{1}{2l_j}\left(\sum_{k\in \Z}(k-n+\frac jp)
      :\elei {k}j\elei{s+n-k-1}{p-j}:\psi\big((k-n)\ep\big)+(s-k-\frac jp)
      :\elei {k}j\elei{s+n-k-1}{p-j}:\psi(k\ep)\right)\\
          &\ +\frac12\dt_{s+n,0}
            \sum_{j\in I}\sum_{-\frac12\leq k<s-\frac12}
          (k+\frac jp)(s-k-\frac jp)\psi(k\ep).
 \end{align*}
 Taking the limit $\ep\rightarrow0$, we get
 \begin{align*}
  [L_s,L_n]=(s-n)\sum_{j\in I}\frac{1}{2l_j}\left(\sum_{k\in \Z}
                                :\elei {k}j\elei{s+n-k-1}{p-j}:\right)
    +2s\dt_{s+n,0}\sum_{j\in I}\frac{j(p-j)}{4p^2}+\dt_{s+n,0}\frac1{12}(s^3-s)|I|,
 \end{align*}
 which amounts to \eqref{eq2.1}.
\end{proof}

\begin{remark}
    In the Appendix, we will give a  vertex algebra  interpretation of Proposition \ref{proprealize}.
   The scalar $\sum_{j\in I}\frac{j(p-j)}{4p^2}$ in \eqref{eq:Lnaction} is in fact obtained by using the formula on the conformal vector action of twisted modules given in \cite[(1.14)]{KW}.
\end{remark}

Set
\[\h_I^+=\spanc{\elei ni,C_i\ \mid\ i\in I, n\in\N},\]
which is a subalgebra of $\h_I$ and a subalgebra of $\g^{(m)}$.
In fact, we have $\h_I^+=\h_I\cap \g^{(m)}$.

\begin{proposition}\label{prophimod}
 Let $\phi$ be a Whittaker function on $\h_I^+$.
 Then the universal Whittaker $\h_I$-module $M_{\h_I,\phi}$ is irreducible
 if and only if  $\phi(C_i)\ne 0$ for all $i\in I$.
\end{proposition}
\begin{proof}
 Recall the direct sum decomposition $\h_I=\oplus_{i\in I_0}\h_{\{i,p-i\}}$ of $\h_I$ from \eqref{eq:dechI}.
 Then we have
 \[
   M_{\h_I,\phi}\cong\otimes_{i\in I_0} M_{\h_{\{i,p-i\}},\phi_i}
 \]
 as $\h_I$-modules, where $\phi_i$ denotes the restriction of $\phi$ on $\h_I^+\cap \h_{\{i,p-i\}}$.
 This implies that the $\h_I$-module $M_{\h_I,\phi}$ is irreducible if and only if so is the $\h_{\{i,p-i\}}$-module $M_{\h_{\{i,p-i\}},\phi_i}$ for every $i\in I_0$.
 Then the assertion follows from the fact that $M_{\h_{\{i,p-i\}},\phi_i}$ is irreducible if and only if $\phi_i(C_i)\neq0$ (see \cite{Chr}).
\end{proof}

Let $\phi$ be a Whittaker function on $\h_I^+$ such that
\begin{eqnarray}\label{eq:cononphi}
    l_i=\phi(C_i)\ne 0\quad \text{and}\quad \phi(I_n^{(i)})=0\quad \text{for all}\ i\in I, n\ge m.
\end{eqnarray}
Extend $\phi$ to a Whittaker function $\phi^e$ on $\g^{(m)}$ by
\begin{equation}\begin{aligned}\label{eq:defphie}
    \phi^e(I_n^{(i)})&=\phi^e(C_i)=\phi^e(L_l)=0,\quad
    \phi^e(C_0)=|I|,\quad \text{ for }i\in I^c, n\in \N, l>2m, \\
    \phi^e(L_n)&=\sum\limits_{j\in I}\frac1{2l_j}\left(\sum\limits_{k=0}^{m-1}
                            \phi\(I_k^{(j)}\)\phi\(I_{n-k-1}^{(p-j)}\)\right),
    \quad \text{ for }m\le n\le 2m.
\end{aligned}
\end{equation}
Applying Proposition \ref{proprealize}, we obtain the following free field realization of certain irreducible Whittaker $\g$-modules.

\begin{proposition} \label{prop:FFR}
 Let $\phi$ be a Whittaker function on $\h_I^+$  satisfying the conditions in \eqref{eq:cononphi}.
 Then there is a $\g$-module structure on the polynomial ring
 \[
  P_I=\C[t_{n,i}\mid n\in \Z_+, i\in I]
 \]
 such that
 \begin{align} \label{eq:ffract}
    I_{-n}^{(i)}&\mapsto t_{n,i},\quad I_{n-1}^{(p-i)}\mapsto \phi(C_i)\left(n-\frac{i}{p}\right)
     \frac{\partial}{\partial t_{n,i}}+\phi(I_{n-1}^{(p-i)}),\quad
     C_i\mapsto \phi(C_i)\quad\text{for}\ i\in I, n\in \Z_+,\\
     I_n^{(j)}&\mapsto 0,\quad C_j\mapsto 0,\quad C_0\mapsto |I|\quad
     \text{for}\ j\in I^C, n\in \Z, \notag
 \end{align}
 and $L_n$ acts as in \eqref{eq:Lnaction} for $n\in \Z$.
 Furthermore, the resulting $\g$-module  is an irreducible Whittaker module of type $\phi^e$.
\end{proposition}
\begin{proof}
 It is straightforward to check that the action \eqref{eq:ffract} makes $P_I$ a Whittaker $\h_I$-module of type $\phi$.
 And by the PBW theorem we have $P_I\cong M_{\h_I,\phi}$ as $\h_I$-modules.
 Then $P_I$ becomes a Whittaker $\g$-module of type $\phi^e$ with action defined in the proposition by Proposition \ref{proprealize}, and is irreducible by \eqref{eq:cononphi} and Proposition \ref{prophimod}.
 This completes the proof.
\end{proof}

\subsection{Certain tensor product  $\g$-modules}\label{subsec:tengmod}
Let $V$ be a restricted $\h_I$-module of generic level, and let $W$ be a $\g_{I^C}$-module.
We regard $V$ as a $\g$-module by Proposition \ref{proprealize}, denoted by $V^{\g}$.
Furthermore, we extend $W$ to a $\g$-module by letting $\h_I W=0$, denoted by $W^e$.
Then we obtain in this way a tensor product $\g$-module
\[
 V^{\g}\otimes W^e.
\]

\begin{proposition}\label{proptengmod}
 Let $V,V'$ be irreducible restricted $\h_I$-modules of generic levels,
 and let $W,W'$ be $\g_{I^C}$-modules.
Then\\
 (1) every $\g$-submodule of $V^{\g}\otimes W^e$ is of the form $V^\g\otimes (W_0)^e$ for some $\gic$-submodule $W_0$ of $W$. In particular,  $V^{\g}\otimes W^e$ is irreducible if and only if the $\gic$-module $W$ is irreducible.\\
 (2)  $V^{\g}\otimes W^e\cong V'^{\g}\otimes W'^e$ if and only if $V\cong V'$ as $\h_I$-modules and $W\cong W'$ as $\gic$-modules.
\end{proposition}
\begin{proof}
 (1) Note that $V^\g$ is irreducible as a $\g$-module and
  $$\h_{I}\subseteq\mathrm{Ann}_\g(S)=\{x\in\g\ \vline\ xS=0\}$$
  for any finite subset $S\subset W^e$.
  Then (1) follows from Lemma \ref{lem2.7} (2).\\
  (2) The sufficiency is clear and we prove the necessity in the following.
  Let $\vf:V^{\g}\otimes W^e\longrightarrow V'^{\g}\otimes W'^e$ be
  a $\g$-module isomorphism and fix $x\in V$.
  For any $v\in W$ write
  $$\vf(x\otimes v)=\sum_{i=1}^ty_i\otimes w_i,$$
  where $t\in\Z_+, w_i\in W'$, $y_i\in V'$, and $y_1,y_2,\dots,y_t$ are linearly independent.
  Then by Lemma \ref{lem2.7} (1), there exists $u_0\in\U(\h_I)$ such that
  $$u_0y_i=\dt_{i,1}y_1,\qquad i=1,2,\dots,t.$$
  Therefore for any $u\in\U(\h_I)$, noting $\h_IW^e=\h_IW'^e=0$, we have
  \begin{align*}
   \vf(uu_0x\otimes v)=\vf(uu_0(x\otimes v))=uu_0\vf(x\otimes v)=uu_0\sum_{i=1}^ty_i\otimes w_i
   =\sum_{i=1}^tuu_0y_i\otimes w_i=uy_1\otimes w_1.
  \end{align*}
  Denote $x_1=u_0x$ and we get
  $$\vf(ux_1\otimes v)=uy_1\otimes w_1\qquad\text{ for any }u\in\U(\h_I).$$
  Note that $V$ is irreducible as an $\h_I$-module.
  Then the linear map $\eta_v:V\longrightarrow V'$ defined by $\eta_v(ux_1)=uy_1$ for any $u\in\U(\h_I)$ is clearly an $\h_I$-module isomorphism.

  Now without loss of generality we may assume that $V=V'$ and $\eta_v$ is the identity.
  Then for any $x\in V, v\in W$ there exists $w_v\in W$ such that
  $$\vf(x\otimes v)=x\otimes w_v.$$
  Define a linear map $\vf':W\longrightarrow W'$ by $\vf'(v)=w_v$, which is clearly bijective.
  Moreover, for any $i\in I^C, n\in\Z$, we have
  \begin{align*}
   \vf(x\otimes Xv)&=\vf(X(x\otimes v))-\vf((Xx)\otimes v)=X\vf(x\otimes v)-Xx\otimes w_v\\
                   &=X(x\otimes w_v)-Xx\otimes w_v=x\otimes Xw_v,
  \end{align*}
  where $X$ represents $L_n$ or $\elei ni$.
  So $\vf'$ is a $\gic$-module isomorphism, proving (2).
\end{proof}

\subsection{Whittaker $\g$-modules of zero level}
The main goal of this subsection is to prove the following result.
\begin{theorem}\label{thm:zerolevel}
 Let $\varphi$ be a Whittaker function on $\g_I^{(m)}$ such that $\varphi(C_i)=0$ for all $i\in I$. Then the universal Whittaker $\g_I$-module $M_{\g_I,\varphi}$ is irreducible if and only if
 $\vfi{m-1}i\neq 0$ for all $i\in I$.
\end{theorem}

By taking $I=\{1,2,\dots,p-1\}$ in Theorem \ref{thm:zerolevel}, we obtain:

\begin{corollary}\label{corzerolevelgmod}
 Let $\varphi$ be a Whittaker function of $\g^{(m)}$ such that $\varphi(C_i)=0$ for all $1\le i\le p-1$. Then the universal Whittaker $\g$-module $M_{\g,\varphi}$ is irreducible if and only if
 $\vfi{m-1}i\neq 0$ for all $1\leq i\leq p-1$.
\end{corollary}

In the rest of this subsection, let $\varphi$ be as in Theorem \ref{thm:zerolevel}.

\begin{lemma}\label{lem:onlyifpart}
  If   $\vfi{m-1}i= 0$ for some  $1\leq i\leq p-1$, then  $M_{\g,\varphi}$ is reducible.
\end{lemma}
\begin{proof}
 Pick an $i\in I$ such that  $\vfi{m-1}i=0$.
Since $\elei{m+k-1}i\vvf=0$ for all $k\geq 0$,
we see that $\elei{-1}i\vvf$ is a Whittaker vector of type $\varphi$,
which generates a proper submodule of $M_{\g,\varphi}$.
\end{proof}

By Lemma \ref{lem:onlyifpart}, it remains to prove the ``if" part of Theorem \ref{thm:zerolevel}.
So from now on we  assume that  $\vfi{m-1}i\neq 0$ for all  $1\leq i\leq p-1$.
Set
\[\gizero=\spanc{L_{n},\elei ni,C_0,C_i\ \mid\ n\in\N, i\in I},\]  which is a subalgebra of $\g_I$ containing $\gim$.

\begin{lemma}\label{lem:conirr}
    If $\varphi(L_k)=0$ for all $k\ge m$, then the induced $\g_I^{(0)}$-module
    \begin{equation}\label{eq:defM0}
          M_0=\Ind_{\gim}^{\gizero}\,\C\vvf
    \end{equation}
    is irreducible.
\end{lemma}
\begin{proof}
For ${\bf{i}}=(i_0,i_1,\dots,i_{m-1})\in \N^m$, write $L^{\bf i}=L_{m-1}^{i_m-1}\cdots L_1^{i_1}L_0^{i_0}$.
Define a total order ``$\prec$" on $\N^m$ by
 $${\bf i} \prec{\bf j}\iff \text{ there exists } 0\leq k\leq m-1 \text{ such that }\(j_s=i_s, 0\leq s<k\) \text{ and }i_k<j_k.$$
For any $w\in M_0$, write $\supp (w)$ for the finite subset of $\N^m$ such that
 $$w=\sum_{{\bf i}\in \supp (w)}a_{\bf i}L^{\bf i}\vvf\quad\text{for some (unique)}\ a_{\bf i}\in \C^\times.$$
Additionally, write  $\deg(w)$ for the maximal element in $\supp(w)$ with respect to $\prec$, called the degree of $w$.

Assume now that $V$ is a nonzero $\g_I^{(0)}$-submodule of $M_0$ and $0\neq w\in V$ with minimal degree. Suppose $w\notin \C\vvf$, write $\deg(w)={\bf j}=(j_0,j_1,\dots,j_{m-1})$, and  $q=\min \{0\le s\le m-1\mid j_s\neq 0\}$.
From the equality
 $$\left(\elei {m-q-1}i-\vf\left(\elei {m-q-1}i\right)\right)w
   =\sum_{\bf k\in\N^m}a_{\bf k}\left[\elei {m-q-1}i,L^{\bf k}\right]\vvf,$$
 it is straightforward to check that
 $$\deg\left(\left(\elei {m-q-1}i-\vf\left(\elei {m-q-1}i\right)\right)w\right)={\bf j}-{\bf\epsilon}_q,$$
 where ${\bf\epsilon}_q$ is the element in $\N^m$ whose $(q+1)$-position (from right) is 1 and 0 elsewhere.
 We obtain a vector in $V$ of degree less than $\deg(w)$, contradicting to the choice of $w$.
 This implies that $w\in \C\vvf$, and so   $V=M_0$ and $M_0$ is irreducible.
\end{proof}

\begin{lemma}\label{propreszerolevelgimod}
Let $V$ be an irreducible $\g_I^{(0)}$-module such that $C_i$ acts trivially for all $i\in I$. Suppose that there is a non-negative integer $k$ satisfying the following two conditions\\
 (1) all $\elei {k}j, \ j\in I$, act injectively on $V$;\\
 (2) $L_nV=\elei {n}jV=0$ for all $n>k$ and $j\in I$.\\
 Then the induced $\g_I$-module $\mathrm{Ind}_{\g_I^{(0)}}^{\g_I}V
  $ is also irreducible.
\end{lemma}
\begin{proof}
    When $I=\{1,2,\dots,p-1\}$ (and so $\g_I=\g$), the assertion is proved in \cite[Theorem 3.4]{GX}.
   However, the proof given therein is also valid for any symmetrical subset $I$ of $\{1,2,\dots,p-1\}$, and we omit the details.
\end{proof}

\begin{lemma}\label{lem:globalconirr}
 If $\varphi(L_k)=0$ for all $k\ge m$,  then the $\g_I$-module $M_{\g_I,\varphi}$ is irreducible.
\end{lemma}
\begin{proof}
Let $M_0$ be as in \eqref{eq:defM0}, which is an irreducible $\g_I^{(0)}$-module by Lemma \ref{lem:conirr}.
Note that
\[C_iM_0=L_nM_0=I_n^{(i)}M_0=0\]
for all $n>m-1$ and $i\in I$. Furthermore, it is easy to see that $I_{m-1}^{(i)}$ acts injectively on $M_0$ for all $i\in I$.
  Thus, by applying Lemma \ref{propreszerolevelgimod} with $k=m-1$, we find that
$$\mgivf=\Ind_{\gim}^{\gi}\C\vvf\cong\Ind_{\gizero}^{\gi}\left(\Ind_{\gim}^{\gizero}\C\vvf\right)
   =\Ind_{\gizero}^{\gi}M_0$$
 is an irreducible $\gi$-module.
\end{proof}

\textbf{Proof of Theorem \ref{thm:zerolevel}}:
 For any $i\in I$, denote by $A_i$ the $(m+1)\times (m+1)$-matrix  whose  $(l,k)$-position is $\vfi{m+k-l-1}i$ (here we write $\vfi{-1}i=0$ symbolically).
 Since $\vf(\elei {m+j-1}i)=0$ for any $j>0$, we see that $A_i$ is lower triangular with all diagonal entries being $\vfi{m-1}i\neq0$, and hence is invertible.
 So there exist $a_{i,0},a_{i,-1},\dots,a_{i,-m}\in\C$ such that
 $$\left(\vfl m,\vfl{m+1},\dots,\vfl{2m}\right)=(a_{i,0},a_{i,-1},\dots,a_{i,-m})A_i.$$
 This gives (noticing that $\vf(C_i)=0$)
 \begin{equation*}
  0=\vfl n-\sum_{k=-m}^0a_{i,k}\vfi{n+k}i
    +\frac1{2| I|}\sum_{k\in\Z}\sum_{l=-m}^0a_{i,k}a_{p-i,-n-k-1}\dt_{n+k+l+1,0}\vf(C_i)
  \ \ \ \quad \text{for all }n\geq m.
 \end{equation*}
By taking the summation over all $i\in I$, we get that for any $n\geq m$
 \begin{equation}\label{eq3.7}
  0=\vfl n-\frac1{| I|}\sum_{i\in I}\sum_{k=-m}^0a_{i,k}\vfi{n+k}i+
    \frac1{2| I|^2}\sum_{i\in I}\sum_{k\in\Z}\sum_{l=-m}^0a_{i,k}a_{p-i,-n-k-1}\dt_{n+k+l+1,0}\vf(C_i).
 \end{equation}
 Set
 $$\al=-\frac1{| I|}\sum_{i\in I}\sum_{k=-m}^0\frac{pa_{i,k}}{pk+i}\elei ki.$$
 Then the Lie algebra automorphism $\eta_\al=\exp(\text{ad}\, \al)$ of $\g_I$ has the following expression:
 \begin{align*}
  \eta_\al(\elei nj)&=\elei nj-\frac1{| I|}\sum_{i\in I}\sum_{k=-m}^0a_{i,k}\dt_{i+j,p}\dt_{n+k+1,0}C_j,\ \ \
      \quad \eta_\al(C_s)=C_s,\\
  \eta_\al(L_n)&=L_n-\frac1{| I|}\sum_{i\in I}\sum_{k=-m}^0a_{i,k}\elei {n+k}i
             +\frac1{2| I|^2}\sum_{i\in I}\sum_{k=-m}^0\sum_{l=-m}^0a_{i,k}a_{p-i,-n-k-1}\dt_{n+k+l+1,0}C_i,
 \end{align*}
 where $j\in I, s\in I\cup\{0\}$ and $n\in\Z$.
 Note that $\eta_\al=\exp(\text{ad}\, \al)$ preserves the subalgebra $
 \g^{(m)}$ and so induces a Whittaker function $\varphi'=\phi\circ \eta|_{\g^{(m)}}$ on $\g^{(m)}$.
 Then one gets a new $\g_I$-module structure on $\mgivf$ defined by
 $$x\cdot\vvf=\eta_\al(x)\vvf\ \ \ \ \ \text{for any }x\in\gi,$$
 Clearly, the resulting module is the universal Whittaker $\g_I$-module $M_{\gi,\vf'}$ of type $\vf'$ and isomorphic to $\mgivf$.
 From \eqref{eq3.7} it follows that $\varphi'(L_k)=0,\ \varphi'(\elei ni)=\varphi(\elei ni)$ and $\varphi'(C_j)=0$ for all $k\geq m, n\in\Z$ and $j\in I\cup\{0\}$.
By Lemma \ref{lem:globalconirr}, the $\g_I$-module $M_{\gi,\vf'}$ is irreducible and so is $M_{\g_I,\varphi}$. This finishes the proof.

\subsection{Structure of Whittaker $\g$-modules}
In this subsection, let $\varphi$ be a fixed Whittaker function on $\g^{(m)}$. Set
\[
l_i=\vf(C_i)\quad \text{for}\ 1\leq i\leq p-1\quad \text{and}\quad
J=\{j=1,2,\dots,p-1\mid l_j\ne 0\}.
\]

Note that we have the direct sum decomposition  \[\g^{(m)}=\g_{J^C}^{(m)}\oplus \h_J^+.\]
Denote by
\[\phi=\varphi|_{\h_J^+}\] the restriction of $\varphi$ on $\h_J^{+}$. Note that $\phi$ satisfies the conditions in \eqref{eq:cononphi} (with $I=J$).
Then there is a Whittaker function $\phi^e$ on $\g^{(m)}$ which extends $\phi$ (see \eqref{eq:defphie}).
Set
\[
\varphi'=\varphi-\phi^e\quad \text{and}\quad
\psi=\varphi'|_{\g_{J^C}^{(m)}}.
\]
Then we have the universal Whittaker $\h_J$-module $M_{\h_J,\phi}$ with a universal Whittaker vector $v_\phi$ (which is irreducible by Proposition \ref{prophimod}),
and the universal Whittaker $\g_{J^C}$-module $M_{\g_{J^C},\psi}$ with a universal Whittaker vector $v_\psi$.
With the notation given in Subsection \ref{subsec:tengmod}, we have a tensor product $\g$-module
\[(M_{\h_J,\phi})^\g\otimes (M_{\g_{J^C},\psi})^e.\]

\begin{proposition}\label{propgmodiso}
The universal Whittaker $\g$-module $M_{\g,\varphi}$ is isomorphic to
$\(M_{\h_J,\phi}\)^\g\otimes \(M_{\g_{J^C},\psi}\)^e$.
\end{proposition}
\begin{proof}
 One can easily check that $\(M_{\h_J,\phi}\)^\g\otimes \(M_{\g_{J^C},\psi}\)^e$ is a Whittaker module
 of type $\vf$ generated by $v_\phi\otimes v_\psi$.
 Therefore, there is a $\g$-module epimorphism $\pi:M_{\g,\vf}\longrightarrow\(M_{\h_J,\phi}\)^\g\otimes \(M_{\g_{J^C},\psi}\)^e$ defined by
 $\pi(\vvf)=v_\phi\otimes v_\psi$.
 Moreover, applying the PBW theorem one sees that $\pi$ is an isomorphism.
\end{proof}

Now applying Propositions \ref{prop2.5}, \ref{prophimod}, \ref{proptengmod}, \ref{propgmodiso} and Theorem \ref{thm:zerolevel}, we obtain the following result.
\begin{theorem}\label{mainthm}
Let $m$ be a positive integer, and let $\varphi$ be a Whittaker function on $\g^{(m)}$. \\
 (1) Assume that $\vf(C_i)\neq 0$ for all $1\leq i\leq p-1$.
 Then the universal Whittaker $\g$-module $M_{\g,\varphi}$ is irreducible if and only if
 $$(\vfl{2m},\vfl{2m-1})\neq \left(0,\ \ \sum_{i=1}^{p-1}\frac1{2\vf(C_i)}\vfi{m-1}i\vfi{m-1}{p-i}\right).$$
 (2) Assume that $\vf(C_i)=0$ for some $1\leq i\leq p-1$. Then the universal Whittaker $\g$-module $M_{\g,\varphi}$ is irreducible if and only if for every $1\le i\le p-1$,
 $$\vfi {m-1}i\neq0\quad \text{ whenever } \vf(C_i)=0.$$
 (3) Every irreducible Whittaker $\g$-module of type $\varphi$
 is isomorphic to  $(M_{\h_J,\phi})^\g\otimes W^e$,
 where $W$ is an irreducible quotient of the $\gjc$-module $M_{\g_{J^C},\psi}$.
\end{theorem}
\begin{remark}
 Whittaker $\g$-modules of level zero were studied in \cite{GX}, where a sketchy proof of the irreducibility of the universal Whittaker $\g$-modules with respect to a Whittaker function $\vf:\g^{(0)}\rightarrow\C$ was given (the result matches with Theorem \ref{mainthm} (2)), and some irreducible quotients of the reducible ones were also discussed (see \cite[Subsection 5.4]{GX}).
\end{remark}

\section{$\U(\C L_0)$-free $\g$-modules}
In this section, we first classify all $\g$-modules that are free of rank 1 when viewed as $\U(\C L_0)$-modules, and then determine their irreducibility.

For every  $\lmd\in\C^\times$ and $\ov\al=(\al_0,\al_1,\dots,\al_{p-1})\in\C^p$, it is straightforward to see that there is a $\g$-module structure on
the polynomial ring $\C[t]$ such that
$$ L_nf(t)=\lmd^n(t+\al_0n)f(t+n),\ \ \ \ \
 \elei nif(t)=\al_i\lmd^nf(t+n+\frac ip),\ \ \ \ \
 C_jf(t)=0, $$
where $n\in\Z, 1\leq i\leq p-1, 0\leq j\leq p-1$, and $f(t)\in\C[t]$.
We shall denote the resulting $\g$-module by $\Omega(\lmd,\ov\al)$.

\begin{theorem}\label{thm3.2}
Let $M$ be a $\g$-module that is free of rank 1 when viewed as a $\U(\C L_0)$-module.
Then $M\cong\Omega(\lmd,\ov\al)$ for some $\lmd\in\C^\times$ and $\ov\al\in\C^p$.
\end{theorem}
\begin{proof}
 Viewing $M$ as a $\V$-module, we see that $M\cong\Omega(\lmd,\al_0)$ for some $\lmd\in\C^\times$ and $\al_0\in\C$ by Proposition \ref{prop2.6}. By identifying the underlying space of $M$ with $\Omega(\lmd,\al_0)=\C[t]$, we have that
 $$L_nf(t)=\lmd^n(t+\al_0n)f(t+n),\ \ C_0f(t)=0\ \ \text{ for all } n\in\Z, f\in\C[t].$$
 Note that $f(L_0)1=f(t)$ and $L_0f(t)=tf(t)$ for all $f(t)\in\C[t]$.

 For each $1\leq i\leq p-1$ and $n\in\Z$, by using induction on $k$ one can show that
 \[\elei nit^k=\elei ni(L_0^k1)=\left(t+n+\frac ip\right)^k\elei ni 1.\]
 This gives
 \begin{equation}\label{eq3.1}
   \elei nif(t)=\elei nif(L_0)1=f\left(t+n+\frac ip\right)\elei ni 1.
 \end{equation}
 By setting $\al_i=\elei 0i1$, it follows from
 \eqref{eq3.1} that
 $$-\frac ip\elei ni1=[L_n,\elei 0i]1=L_n\elei 0i1-\elei 0i L_n 1
        =\al_i\lmd^n(t+\al_0n)-\elei 0i(t+\al_0n)=-\frac ip\lmd^n\alpha_i.$$
 Therefore, we obtain that
 $$\elei nif(t)=\al_i\lmd^nf\left(t+n+\frac ip\right).$$

 Now, to show that $M=\Omega(\lmd,\ov\al)$ with $\ov\al=(\al_0,\al_1,\dots,\al_{p-1})$,
 it remains to prove that $C_i$ acts trivially on $M$ for all $1\le i\le p-1$.
 Indeed, if $\elei ni 1=0$ for some $n\in\Z$,
 then $\elei ni M=0$ by \eqref{eq3.1}.
 So we have $\elei mi M=0$ for all $m\in\Z$ by using the bracket $[\elei ni,L_{m-n}]=\left(n+\frac ip\right)\elei mi$.
 This implies that $C_iM=[\elei ni,\elei {-n-1}{p-i}]M=0$.

 Suppose $\elei ni 1\neq0$ for all $n\in\Z$.
 We claim that in this case $\elei ni 1\in\C^\times$ for all $n\in\Z$, which forces $C_iM=0$.
 Assume otherwise that there exists $N\in\Z$ such that $\elei Ni 1=\sum_{j=0}^{k_N}a_jt^j$ with $k_N>0, a_{k_N}\neq 0$.
 Choose $s\in\Z$ such that $s+N+1\neq 0$ and $\left(s+1-\frac ip\right)\left(N+\frac ip\right)<0$.
 Write $X(t)=\elei Ni 1$ and $Y(t)=\elei s{p-i}1=\sum_{j=0}^{l_s}b_jt^j$ with $l_s\geq0, b_{l_s}\neq 0$.
 Since $[\elei Ni,\elei s{p-i}]=0$, we have by \eqref{eq3.1} that
 \begin{align*}
  0&=\elei Ni\elei s{p-i}1-\elei s{p-i}\elei Ni 1
      =X(t)Y\left(t+N+\frac ip\right)-X\left(t+s+1-\frac ip\right)Y(t)\\
   &=a_{k_N}b_{l_s}\left(l_s\left(N+\frac ip\right)-k_N\left(s+1-\frac ip\right)\right)t^{k_N+l_s-1}
     +\text{lower terms},
 \end{align*}
 which contradicts to the choice of $s$. This completes the proof.
\end{proof}

Write $\ov 0$ for the element in $\C^p$ whose components are all $0$.

\begin{proposition}\label{prop3.1}
 (1) The $\g$-module $\Omega(\lmd,\ov0)$ has an unique irreducible submodule $t\Omega(\lmd,\ov0)$.\\
 (2) The $\g$-module $\Omega(\lmd,\ov\al)$ is irreducible if and only if $\ov\al\neq\ov0$.
\end{proposition}
\begin{proof}
 The first assertion follows from Proposition \ref{prop2.6} (1).
 For the second one, we only need to show that $\Omega(\lmd,\ov\al)$ is irreducible whenever  $\ov\al\neq\ov0$. If $\al_0\neq 0$, then $\Omega(\lmd,\ov\al)$ is irreducible by Proposition \ref{prop2.6} (2). Assume now that  $\al_i\ne 0$ for some $1\le i\le p-1$,
 and take a nonzero $\g$-submodule $V$ of $\Omega(\lmd,\ov\al)$.
 Let $f(t)\in V$ be a nonzero polynomial with minimal degree $s$. If $s>0$, then
 $$f(t)-f\(t+\frac ip\)=f(t)-\al_i^{-1}\elei 0if(t)\quad \(\in V\setminus\{0\}\)$$
 is of degree $s-1$, a contradiction. This forces that $f(t)$ has degree $0$ and so $1\in V$.
 Applying $L_0$ to $1$ repeatedly, one gets $V=\Omega(\lmd,\ov\al)$, and $\Omega(\lmd,\ov\al)$ is irreducible as required.
\end{proof}

\section{Tensor product of $\U(\C L_0)$-free modules and restricted modules}

In this section we study tensor products of $\U(\C L_0)$-free modules $\Omega(\lmd,\ov\al)$ and irreducible restricted modules.
Specifically, we determine the irreducibility of the tensor product $\g$-modules $\Omega(\lmd,\ov\al)\otimes M$
for any $\lmd\in\C^\times, \ov\al\in\C^p$ and any irreducible restricted $\g$-module $M$,
and the necessary and sufficient condition for any two such tensor products to be isomorphic.

\begin{proposition}\label{prop5.1}
 Let $\lmd\in\C^\times$, $\ov\al=(\al_0,\al_1,\dots,\al_{p-1})\in\C^p$, and let $M$ be an irreducible restricted $\g$-module. \\
 (1) If $\ov\al\neq \ov0$, then the $\g$-module $\Omega(\lmd,\ov\al)\otimes M$ is irreducible.\\
 (2) The $\g$-module $t\Omega(\lmd,\ov0)\otimes M$ is irreducible.
\end{proposition}
\begin{proof}
 We will only prove the first assertion, the second one can be proved in a similar way and we omit the details.

 For any $v\in M$, choose $N(v)\in\N$ such that $L_nv=\elei ni v=0$ for all $n\geq N(v)$ and $1\leq i\leq p-1$.
 Let $V$ be a nonzero $\g$-submodule of $\Omega(\lmd,\ov\al)\otimes M$.
 Let $w=\sum_{i=0}^st^i\otimes v_i$ be a nonzero vector in $V$  with $s\in \N, v_i\in M, v_s\neq 0$ and $s$ minimal.
 Set $N=\max\{N(v_i)\mid i=0,1,2,\dots,s\}$. Then
 $$L_nv_i=\elei njv_i=0\quad\text{ for all }n\geq N, 1\leq j\leq p-1, 0\leq i\leq s.$$

 We claim that $1\otimes v_s\in V$.
 The proof of this claim is divided into two cases:

 {\em Case 1}: $\al_i\neq 0$ for some $1\leq i\leq p-1$. For any $n\geq N$, we have
 $$\al_i^{-1}\lmd^{-n}\elei niw=\sum_{j=0}^s\left(t+n+\frac ip\right)^j\otimes v_j=
      \sum_{j=0}^s\left(n+\frac ip\right)^jw_j\in V,$$
 where $w_j\in\Omega(\lmd,\ov\al)\otimes M$.
 By taking $n=N, N+1,N+2,\dots,N+s$, we get $(s+1)$ equations with indeterminants $w_0,w_1,\dots,w_s$.
 Since the coefficient matrix  is an invertible Vandermonde matrix,
 we find that $w_j\in V$ for all $0\le j\le s$.
 In particular, $w_s=1\otimes v_s\in V$.

 {\em Case 2}: $\al_i=0$ for all $1\leq i\leq p-1$ and $\al_0\neq 0$.
 By  applying $L_n$ ($n>N$) to $w$,
 a similar argument as that in Case 1 shows that $1\otimes v_s\in V$.

 Now we are ready to prove that  $\Omega(\lmd,\ov\al)\otimes v_s\subseteq V$.
Let $k\geq 0$ and
 suppose  $t^j\otimes v\in V$ for all $0\leq j\leq k$.
 From the fact
 \[L_n(t^k\otimes v_s)=\lmd^n(t+\al_0n)(t+n)^k\otimes v_s\quad \text{for}\  n\geq N(v),\]  one concludes that
 $$t(t+n)^k\otimes v_s=\lmd^{-n}L_n(t^k\otimes v_s)-\al_0n(t+n)^k\otimes v_s\in V.$$
 This gives that $t^{k+1}\otimes v_s\in V$, as desired.

 Write $W=\{w\in M\mid \Omega(\lmd,\ov\al)\otimes w\in V\}$, which is a nonzero subspace of $M$.
 For every $w\in W, n\in\Z$, and $1\leq i\leq p-1$, we have
 $$\Omega(\lmd,\ov\al)\otimes xw=x\left(\Omega(\lmd,\ov\al)\otimes
          w\right)-\left(x\Omega(\lmd,\ov\al)\right)\otimes w\in V,$$
 where $x$ denotes $L_n$ or $\elei ni$.
 This shows that $W$ is a submodule of $M$.
 Therefore, $W=M$, $V=\Omega(\lmd,\ov\al)\otimes M $ and the $\g$-module $\Omega(\lmd,\ov\al)\otimes M$ is irreducible.
\end{proof}

Now we determine the isomorphism classes of the tensor products of $\U(\C L_0)$-free modules and irreducible restricted modules.
\begin{proposition}\label{prop5.2}
 Let $\lmd,\mu\in\C^\times, \ov\al,\ov\beta\in\C^p$ and $V,W$ be irreducible restricted modules over $\g$.
 Then  $\Omega(\lmd,\ov\al)\otimes V\cong\Omega(\mu,\ov\beta)\otimes W$
 if and only if $\lmd=\mu, \ov\al=\ov\beta$, and $V\cong W$.
\end{proposition}
\begin{proof}
 It suffices to prove the ``only if" part, and for that we take a $\g$-module isomorphism
 \[\vf:\Omega(\lmd,\ov\al)\otimes V\longrightarrow\Omega(\mu,\ov\beta)\otimes W.\]
 For any $v\in V$, write
 \begin{equation}\label{eq:exvf}
    \vf(1\otimes v)=\sum_{i=0}^st^i\otimes w_i
 \end{equation} with $w_i\in W,s\geq 0$ and $w_s\neq 0$.
 Take $N_v\in\Z_+$ such that $L_nv=L_n(w_i)=\elei njv=\elei njw_i=0$ for all $n\geq N_v,1\leq j\leq p-1,0\leq i\leq s$.

 We first show that $\lambda=\mu$.
 For any $m,n\geq N_v$, by combining \eqref{eq:exvf} with the equality
 \[\vf\left(\left(\lmd^{-m}L_m-\lmd^{-n}L_n\right)(1\otimes v)\right)
     =\left(\lmd^{-m}L_m-\lmd^{-n}L_n\right)\vf(1\otimes v),\]
     we get that
 \begin{equation}\label{eq5.1}
   \al_0(n-m)\sum_{i=0}^st^i\otimes w_i
   =\left(\frac\mu\lmd\right)^m\sum_{i=0}^s(t+\be_0m)(t+m)^i\otimes w_i
     -\left(\frac\mu\lmd\right)^n\sum_{i=0}^s(t+\be_0n)(t+n)^i\otimes w_i.
 \end{equation}
 Notice that the coefficient of $t^{s+1}\otimes w_s$ in the right hand side of \eqref{eq5.1} is
 $\left(\frac\mu\lmd\right)^m-\left(\frac\mu\lmd\right)^n$, while that in the left hand side is 0.
This implies that  $\lmd=\mu$.

Next we prove that $\ov\alpha=\ov\beta$.
Since  $\lmd=\mu$, \eqref{eq5.1} can be rewritten as
 \begin{equation}\label{eq5.2}
   \al_0(n-m)\sum_{i=0}^st^i\otimes w_i=\sum_{i=0}^s(t+\be_0m)(t+m)^i\otimes w_i
          -\sum_{i=0}^s(t+\be_0n)(t+n)^i\otimes w_i.
 \end{equation}
By comparing  the coefficients of $t^s\otimes w_s$ in the both sides of \eqref{eq5.2}, we find that
 $\al_0(n-m)=(\be_0+s)(n-m)$ and so $\al_0=\be_0+s\geq \be_0$.
 Replacing $\vf$ by $\vf^{-1}$ in the above discussion, we also get $\be_0\geq \al_0$.
 This gives  $\be_0=\al_0$ and $s=0$. Particularly, we have  \[\vf(1\otimes v)=1\otimes w_0.\]
Furthermore, for every $n\geq N_v $ and $1\leq i\leq p-1$, we obtain
 $$\be_i\lmd^n(1\otimes w_0)=\elei ni(1\otimes w_0)=\elei ni\vf(1\otimes v)
      =\vf\left(\elei ni(1\otimes v)\right)=\al_i\lmd^n(1\otimes w_0).$$
This shows $\be_i=\al_i$, as desired.

Finally, we show that $V\cong W$.
Let us define a linear map $\eta:V\longrightarrow W$ by
 $$\vf(1\otimes w)=1\otimes \eta(w)\qquad\text{ for any }w\in V.$$
 It is clear that $\eta$ is injective.
 For every $n\in\Z$ and $1\leq i\leq p-1$,
it follows from the equality $\vf\left(\elei ni(1\otimes v)\right)=\elei ni\vf(1\otimes v)$ that
 $$\al_i\lmd^n\vf(1\otimes v)+\vf\left(1\otimes \elei ni v\right)=\elei ni(1\otimes\eta(v))
   =\al_i\lmd^n(1\otimes\eta(v))+1\otimes\elei ni\eta(v).$$
This implies that $1\otimes\eta\left(\elei ni v\right)=\vf\left(1\otimes \elei ni v\right)=1\otimes\elei ni\eta(v)$, and hence
 $\eta\left(\elei ni v\right)=\elei ni\eta(v)$.

 Similarly, the equality $\vf(L_n(1\otimes v))=L_n\vf(1\otimes v)$ for any $n\geq N_v$ implies that
 $$\vf(t\otimes v)=t\otimes \eta(v).$$
Additionally, by applying the equality  $\vf(L_n(1\otimes v))=L_n\vf(1\otimes v)$ for any $n\in \Z$  we have
 $\vf(1\otimes L_nv)=1\otimes L_n\eta(v)$. Thus
 $$\eta(L_nv)=L_n\eta(v)\qquad\text{ for any }n\in\Z.$$
 This proves that $\eta$ is a $\g$-module homomorphism.
 Since $\eta(V)\neq 0$ and $W$ is irreducible, $\eta$ is an isomorphism.
 We complete the proof.
\end{proof}

Note that Whittaker modules are restricted.
Then we have the following two corollaries by Proposition \ref{prop3.1}, Theorem \ref{mainthm},
Proposition \ref{prop5.1} and Proposition \ref{prop5.2}.

\begin{corollary}
 Let $m\in\Z_+$, $\lmd\in\C^\times$, and $\ov\al\in\C^p$. Let $ \vf$ be a Whittaker function on $\g^{(m)}$, and set $J=\{1\leq i\leq p-1\mid \vf(C_i)\neq 0\}$. \\
 (1) If $J=\{1,2,\dots,p-1\}$, then the $\g$-module $\Omega(\lmd,\ov\al)\otimes M_{\g,\varphi}$ is irreducible if and only if
 $$\ov\al\neq\ov0\quad\text{ and }\quad
   \left(\vfl {2m},\ \ \vfl{2m-1} - \sum_{i=1}^{p-1}\frac1{2\vf(C_i)}\vfi{m-1}i\vfi{m-1}{p-i}\right)\neq (0,0).$$
 (2) If $J\neq\{1,2,\dots,p-1\}$, then the $\g$-module $\Omega(\lmd,\ov\al)\otimes M_{\g,\varphi}$ is irreducible if and only if
 $$\ov\al\neq\ov0\quad\text{ and }\quad\vfi {m-1}i\neq 0\quad\text{for all }i\notin J.$$
\end{corollary}
\begin{corollary}
 Let $m,n\in\Z_+$, $\lmd,\mu\in\C^\times$, $\ov\al,\ov\be\in\C^p$, $\vf$ a Whittaker function on $\g^{(m)}$, and $\phi$ a Whittaker function on $\g^{(n)}$. Let $V_\varphi$ and $V_\phi$ be irreducible Whittaker $\g$-modules of type $\varphi$ and $\phi$, respectively.
 Then $\Omega(\lmd,\ov\al)\otimes V_\varphi\cong\Omega(\mu,\ov\be)\otimes V_\phi$ if and only if
 $(\lmd,\ov\al)=(\mu,\ov\be)$ and $V_\varphi\cong V_\phi$.
\end{corollary}

\section{Appendix}

In this appendix we give a realization of the gap-$p$ Virasoro algebra as a twisted vertex Lie algebra, and present a slightly more generalized version of Proposition \ref{proprealize} (see Proposition \ref{main}).

\subsection{Vertex Lie algebras $\mathcal{N}_J$}
Let
\[\mathcal{N}=\oplus_{m\in \Z}\(\C T(m) \oplus
(\oplus_{i=1}^{p-1} \C N_i(m))\)\oplus\(\oplus_{j=0}^{[\frac{p}{2}]}\C K_j(-1)\)  \] be a Lie algebra,
where $K_0(-1),K_1(-1),\dots,K_{[\frac{p}{2}]}(-1)$ are central and
\begin{equation}\label{relaaffvir}
	\begin{aligned}
	&[T(m),T(n)]=(m-n)T(m+n)+\frac{1}{12}(m^{3}-m)\delta_{m+n,0}K_{0}(-1);\\
		&[T(m),N_{i}(n)]=-n N_{i}(m+n),\quad
[N_{i}(m),N_{j}(n)]=m\delta_{i+j,p}\delta_{m+n,0}K_{i}(-1),
	\end{aligned}
\end{equation}
for $m,n\in\Z$, $i,j=1,\ldots,p-1$, and
$K_{i}(-1):=K_{p-i}(-1)$ if $i>[\frac{p}{2}]$.
For convenience, we also write $N_0(m)=T(m-1)$ for $m\in \Z$.

In terms of generating functions
\[
N_i(z)=\sum_{m\in \Z} N_i(m) z^{-m-1}\quad \text{and}\quad
K_i(z)=K_i(-1)z^0\quad (i=0,1,\dots,p-1),
\]
the Lie relations in \eqref{relaaffvir} can be rewritten as follows:
\begin{equation*}
\begin{split}
 [N_0(z),N_0(w)]&=\(\frac{\p}{\p w}N_0(w)\)z^{-1}\delta\(\frac{w}{z}\)
      +2 N_0(w)\frac{\p}{\p w}z^{-1}\delta\(\frac{w}{z}\)
      +\frac{1}{12}K_0(w)\(\frac{\p}{\p w}\)^3z^{-1}\delta\(\frac{w}{z}\),\\
 [N_0(z),N_i(w)]&=\(\frac{\p}{\p w}N_i(w)\)z^{-1}\delta\(\frac{w}{z}\)
     +N_i(w)\frac{\p}{\p w}z^{-1}\delta\(\frac{w}{z}\),\\
 [N_i(z),N_j(w)]&=\delta_{i+j,p}K_i(w)\frac{\p}{\p w}z^{-1}\delta\(\frac{w}{z}\),
\end{split}
\end{equation*}
for all $i,j=1,2,\dots,p-1$.

Throughout this appendix, let $J$ be a subset of $\{0,1,\dots,p-1\}$ such that  $p-i\in J$ whenever $ i\in J\setminus\{0\}$.
Denote by $\mathcal{N}_J$ the subalgebra of $\mathcal{N}$ with a basis
\begin{equation*} 
 \left\{N_j(n),\ K_i(-1)\mid  n\in \Z, j\in J, i\in J\cap \left\{0,1,2,\dots,[\frac{p}{2}]\right\}\right\},
\end{equation*}
which is a vertex Lie algebra (see \cite[Definition 2.3]{BL}).

Form the induced $\mathcal{N}_J$-module
\begin{equation}\label{defmnj}
 M_{\mathcal{N}_J}=\U(\mathcal{N}_J)\otimes_{\U\(\mathcal{N}^+_J\)}\, \C,
\end{equation}
where $\C$ denotes the one-dimensional trivial module for the algebra
\[
 \mathcal{N}_J^+=\text{Span}_\C \{N_j(n) \mid j\in J, n\in \N \}.
\]
It is known from \cite{DLM, P} that there is a unique vertex algebra structure on $M_{\mathcal{N}_J}$ such that ${\bm 1}=1\otimes 1$ is the vacuum vector and the vertex operators are given by
\[
Y(N_j(-1){\bm 1},z)=N_j(z),\quad Y(K_j(-1){\bm 1},z)=K_j(z)\quad \text{for }j\in J.
\]

Set
\[J^*=J\setminus\{0\},\quad J^C=\{0,1,\dots,p-1\}\setminus J\quad \text{and}\quad J_0=J\cap \left\{1,2,\dots,[\frac{p}{2}]\right\}.\]
Let $c\in \C$, and let $\underline{\ell}=(\ell_j)_{j\in J^*}\in \C^{|J^*|}$ be a $|J^*|$-tuple such that  $\ell_j=\ell_{p-j}$ for any $j\in J^*$.
If $0\notin J$, we write $I_{\mathcal{N}_J}(\underline{\ell})$ for
the ideal of $M_{\mathcal{N}_J}$ generated by the set
\[
\{ K_j(-1){\bm 1}-\ell_j{\bm 1}\mid j\in J\},\] and
form  the quotient vertex algebra
\[
M_{\mathcal{N}_J}(\underline{\ell})=M_{\mathcal{N}_J}/ I_{\mathcal{N}_J}(\underline{\ell}).
\]
Specially, $M_{\mathcal{N}_{\{j,p-j\}}}(\ell_j)$ is the Heisenberg vertex algebra of level $\ell_j$ for every $j\in J^*$.
If $\ell_j\ne 0$, then $M_{\mathcal{N}_{\{j,p-j\}}}(\ell_j)$ is a simple vertex operator algebra
with a conformal vector
\[\omega_{j}(\ell_j)=\frac{1}{2\ell_j}(N_j(-1) N_{p-j}(-1){\bm 1}+N_{p-j}(-1)N_j(-1) {\bm 1}).\]
Moreover, if $J=J^*$ (i.e., $0\notin J$) and $\ell_j\ne 0$ for all $j\in J$,
then $M_{\mathcal{N}_J}(\underline{\ell})$ becomes a vertex operator algebra of central charge $|J|$ with a conformal vector
\[
\omega_J(\underline{\ell})=\sum_{j\in J} \frac{1}{2\ell_j}
(N_j(-1) N_{p-j}(-1){\bm 1}).
\]
In fact, we have an identification
\[
 M_{\mathcal{N}_{J}}(\underline{\ell})= \otimes_{j\in J_0} M_{\mathcal{N}_{\{j,p-j\}}}(\ell_j)
\]
as vertex operator algebras such that
\[
\omega_{J}(\underline{\ell})=\otimes_{j\in J_0} \omega_j(\ell_j).
\]

If $0\in J$, we write $I_{\mathcal{N}_J}(c,\underline{\ell})$ for
the ideal of $M_{\mathcal{N}_J}$ generated by
\[
\left\{K_0(-1){\bm 1}-c{\bm 1},\ K_j(-1){\bm 1}-\ell_j{\bm 1}\mid j\in J^*\right\},\] and
form the quotient vertex algebra
\[
M_{\mathcal{N}_J}(c,\underline{\ell})=M_{\mathcal{N}_J}/ I_{\mathcal{N}_J}(c,\underline{\ell}).
\]
In particular, $M_{\mathcal{N}_{\{0\}}}(c)$ is the universal Virasoro vertex algebra of central charge $c$.
When $J=\{0,1,2,\dots,p-1\}$ (i.e., $\CN_J=\CN$), we have the quotient vertex algebra
\[
M_{\mathcal{N}}(c,\underline{\ell})=M_{\mathcal{N}}/ I_{\mathcal{N}}(c,\underline{\ell}).
\]

Let $\bm{\ell}=(\ell_1,\ell_2,\dots,\ell_{p-1})\in \C^{p-1}$ be such that $\ell_j=\ell_{p-j}$ for all $1\le j\le p-1$. Set
\[
J_{\bm \ell}=\{1\le j\le p-1\mid \ell_j\ne 0\}\quad \text{and}\quad
{\bm \ell}^*=(\ell_j)_{j\in J_{\bm \ell}}\in \C^{|J_{\bm \ell}|}.
\]
Due to the above discussion, we have the following straightforward result.

\begin{proposition}\label{vaiden}
For any $\bm{\ell}\in \C^{p-1}$ and $ c\in \C$,
we have the identification
\begin{equation} \label{voaid}
 M_{\mathcal{N}}(c,\bm{\ell})=M_{\mathcal{N}_{J_{\bm\ell}^c}}(c-|J_{\bm{\ell}}|,{\un 0})\otimes M_{\mathcal{N}_{J_{\bm{\ell}}}}(\bm{\ell}^*)
\end{equation}
of vertex algebras such that
\begin{equation}\begin{split}\label{conformaliden}
   N_0(-1){\bm 1}&= N_0(-1){\bm 1}\otimes \omega_{J_{\bm\ell}}({\bm\ell}^*),\\
   N_i(-1){\bm 1}&=N_{i}(-1){\bm 1}\otimes {\bm 1}\quad \text{if}\ i\notin J_{\bm\ell},\\
   N_i(-1){\bm 1}&= {\bm 1}\otimes N_{i}(-1){\bm 1}\quad \text{if}\ i\in J_{\bm\ell},
\end{split}
\end{equation}
where $i=1,2,\dots,p-1$ and $\un{0}$ is the $|J_{\bm\ell}^C|$-tuple with all components being zero.
\end{proposition}

\subsection{$\sigma_J$-twist of $\mathcal{N}_J$}

Let $\sigma$ be the Lie algebra automorphism of $\CN$ defined by
\[
\sigma(N_i(m))=\xi^i N_i(m)\quad \text{and}\quad
\sigma(K_i(-1))=K_i(-1)
\]
for $i=0,1,\dots,p-1$ and $m\in \Z$,
where $\xi$ is a primitive $p$-th root of unity.
For any $c\in \C$ and ${\bm\ell}=(\ell_1,\ell_2,\dots,\ell_{p-1})\in \C^{p-1}$,
there is an automorphism of the quotient vertex algebra $M_\CN(c,\bm\ell)$,
still denoted by $\sigma$, such that
\[
\sigma(N_j(-1){\bm 1})=\xi^j N_j(-1){\bm 1}\quad\text{and}\quad
\sigma(K_j(-1){\bm 1})=K_j(-1){\bm 1}=l_j\bm1 \quad\text{for all } 0\le j\le p-1,
\]
where we have abused the symbol $\bm1$ for the vacuum vector in $M_\CN(c,\bm\ell)$.

For any subset $J$ of $\{0,1,2,\dots,p-1\}$, denote by $\sigma_J$ the restriction of $\sigma$ onto $\CN_J$,
which preserves the subalgebra $\CN_J^+$.
Thus there is a unique automorphism, still denoted by $\sigma_J$,
of the vertex algebra $M_{\CN_J}$ such that
\[
\sigma_J(N_j(-1){\bm 1})=\xi^j N_j(-1){\bm 1}\quad\text{and}\quad
\sigma_J(K_j(-1){\bm 1})=K_j(-1){\bm 1} \quad\text{for all } j\in J.
\]
Furthermore, it induces an automorphism of the quotient vertex algebra
$M_{\mathcal{N}_{J_{\bm\ell}^c}}(c-|J_{\bm{\ell}}|,{\un 0})$, denoted by $\sigma_{J_{\bm{\ell}}^c}$,
if $0\in J$,
and an automorphism of $M_{\mathcal{N}_{J_{\bm{\ell}}}}(\bm{\ell}^*)$,
denoted by $\sigma_{J_{\bm{\ell}}}$, if $0\notin J$.

\begin{lemma}\label{autiden}
Under the identification in \eqref{voaid}, we have
\[\sigma=\sigma_{J_{\bm{\ell}}^c}\otimes \sigma_{J_{\bm{\ell}}}.\]
as automorphisms of vertex algebras.
\end{lemma}

On the other hand, associated to the automorphism $\sigma$, we have the  $\sigma$-twist $\CN(\sigma)$ of $\CN$ from \cite{BL}.
Specifically, $\CN(\sigma)$ is a Lie algebra with a basis
\[
\left\{\bar{N}_i(m+\frac{i}{p}), \bar{K}_j(-1) \mid m\in \Z, i=0,1,\dots,p-1, j=0,1,\dots, [\frac{p}{2}]\right\},
\]
and Lie relations (in terms of generating functions)
\begin{equation*}\begin{split}
[\bar{N}_0(z),\bar{N}_0(w)]&=\(\frac{\p}{\p w}\bar{N}_0(w)\)z^{-1}\delta\(\frac{w}{z}\)
+2\bar{N}_0(w)\frac{\p}{\p w}z^{-1}\delta\(\frac{w}{z}\)
 +\frac{1}{12}\bar{K}_0(w)\(\frac{\p}{\p w}\)^3z^{-1}\delta\(\frac{w}{z}\),\\
[\bar{N}_0(z),\bar{N}_i(w)]&=\(\frac{\p}{\p w}\bar{N}_i(w)\)z^{-1}\delta\(\frac{w}{z}\)+\bar{N}_i(w)\frac{\p}{\p w}z^{-1}\delta\(\frac{w}{z}\),
\\
[\bar{N}_i(z),\bar{N}_j(w)]&=\delta_{i+j,p}\bar{K}_i(w)\frac{\p}{\p w}z^{-1}\(\frac{w}{z}\)^{\frac{i}{p}}\delta\(\frac{w}{z}\),
\end{split}
\end{equation*}
for $i,j=1,2,\dots,p-1$, where the generating functions are
\[
\bar{N}_i(z)=
\sum_{m\in \Z}\bar{N}_i(m+\frac{i}{p})z^{-m-\frac{i}{p}-1},\quad
\bar{K}_i(z)=\bar{K}_i(-1)z^0,
\]
for $i=0,1,\dots,p-1$ and $\bar{K}_i(-1)=\bar{K}_{p-i}(-1)$ if $i>[\frac{p}{2}]$.

Moreover, we have the $\sigma_J$-twist $\CN_J(\sigma_J)$ of $\CN_J$,
which is exactly the subalgebra of $\CN(\sigma)$ spanned by
\[
\left\{\bar{N}_j(m+\frac{j}{p}), \bar{K}_j(-1) \mid m\in \Z, j\in J\right\}.
\]
We say that an $\CN_J(\sigma_J)$-module $W$ is restricted if for every $w\in W$ and $j\in J$, $\bar{N}_j(m+\frac{j}{p}).w=0$ for any sufficiently large $m$.

Recall the vertex algebra $M_{\CN_J}$ from \eqref{defmnj}.
The following result is standard (cf. \cite{Li,BL}).
\begin{proposition}\label{chartwistmod}
Every (weak) $\sigma_J$-twisted $M_{\CN_J}$-module $(W,Y_W(\cdot,z))$ is a restricted $\CN_J(\sigma_J)$-module such that
\[
\bar{N}_i(z)=Y_W(N_i(-1)\bm{1},z),\quad \bar{K}_i(z)=Y_W(K_i(-1)\bm{1},z)\quad\text{for any }i\in J.
\]
On the other hand, every  restricted $\CN_J(\sigma_J)$-module $W$  is a (weak) $\sigma_J$-twisted $M_{\CN_J}$-module such that
\[
Y_W(N_i(-1)\bm{1},z)=\bar{N}_i(z),\quad Y_W(K_i(-1)\bm{1},z)=\bar{K}_i(z)\quad\text{for any }i\in J.
\]
\end{proposition}

Let $U$ be a vector space, $j=0,1,\dots,p-1$, and
\[a(z)=\sum_{n\in \Z+\frac{j}{p}} a(n)z^{-n-1}\in  \mathrm{End}(U)[[z^{\pm \frac{1}{p}}]].\]
We split $a(z)=a(z)_++a(z)_-$ as follows:
\[
a(z)_+:=\sum_{n<0} a(n)z^{-n-1}\ \quad\text{and}\quad
a(z)_-:=\sum_{n\ge 0} a(n)z^{-n-1}.
\]
Then we have the usual normal ordered product
\[
:a(z)b(w):=a(z)_+b(w)+b(w)a(z)_-,
\]
where $b(z)$ is another element in $\mathrm{End}(U)[[z^{\pm \frac{1}{p}}]]$.

\begin{lemma}
Assume that $J=J^*$ and $\underline{\ell}\in \C^{|J|}$ is generic.
Then for every $\sigma_J$-twisted $M_{\CN_J}(\underline{\ell})$-module $V$, we have
\begin{equation}\label{key}
Y_V(\omega_{J}(\underline{\ell}),z)=\sum_{j\in  J}\frac{1}{2\ell_j}\(:Y_V(N_j(-1){\bm 1},z) Y_V(N_{p-j}(-1){\bm 1},z):
-\ell_j{{\frac{j}{p}}\choose{2}}z^{-2}\).
\end{equation}
\end{lemma}
\begin{proof}
The assertion follows from \cite[(1.14)]{KW}.
\end{proof}

\subsection{A generalization of Proposition \ref{proprealize}}

Let $\theta: \CN(\sigma)\rightarrow \g$ be the linear map defined by
\[
\bar{N}_0(m)\mapsto L_{m-1},\quad \bar{N}_i(m+\frac{i}{p})\mapsto I_{m}^{(i)},\quad
\bar{K}_j(-1)\mapsto C_j,
\]
for $m\in \Z$, $i=1,2,\dots,p-1$, and $j=0,1,\dots,[\frac{p}{2}]$.
The following result gives a realization of $\g_J$ as a twisted vertex Lie algebras whose proof is straightforward.

\begin{lemma}\label{LAiso}
The map $\theta$ is a Lie algebra isomorphism. Furthermore,
 $\theta$ induces an isomorphism
\[
\theta_J=\theta|_{\CN_J(\sigma_J)}:\ \CN_J(\sigma_J)\rightarrow \g_J\]
when $J\ne J^*$, and an isomorphism
\[
\theta_J=\theta|_{\CN_J(\sigma_J)}:\ \CN_J(\sigma_J)\rightarrow \h_J\]
when $J=J^*$.
\end{lemma}

For any nonzero $\bm{\ell}\in\C^{p-1}$ and $c\in \C$, in the following proposition we give a construction of (irreducible) restricted $\g$-modules of central charge $c$ and of level $\bm{\ell}$.
This is a slightly generalized version of Proposition \ref{proprealize}.

\begin{proposition}\label{main}
 Let $c\in \C$, $\bm{\ell}=\{\ell_1,\dots,\ell_{p-1}\}$ be a nonzero $(p-1)$-tuple in $\C^{p-1}$ such that $l_j=l_{p-j}$ for all $1\le j\le p-1$, and
 denote $J_{\bm\ell}=\{j\mid \ell_j\neq 0\}$ and $\bm\ell^*=(l_j)_{j\in J_{\bm\ell}}$.
 Let $W$ be a restricted $\g_{J^c_{\bm{\ell}}}$-module of central charge $c-|J_{\bm{\ell}}|$ and of level zero,
 and let $V$ be a restricted $\h_{J_{\bm{\ell}}}$-module of level $\bm{\ell}^*$.
 Then $W\otimes V$ becomes a restricted $\g$-module of central charge $c$ and of level $\bm{\ell}$ with action
 \begin{equation*}\begin{split}
  L_k.(w\otimes v)&=(L_k.w)\otimes v+ w\otimes (L^*_k.v),\\
  I_k^{(i)}.(w\otimes v)&= \begin{cases} w\otimes (I_k^{(i)}.v)\ &\text{if}\ i\in J_{\bm\ell}\\
  (I_k^{(i)}.w)\otimes v\ &\text{if}\ i\notin J_{\bm\ell}
 \end{cases},
 \end{split}\end{equation*}
 where $w\in W, v\in V$, $k\in \Z$, $i=1,2,\dots,p-1$,  and
 \begin{eqnarray*}
  L^*_k&=& \sum_{j\in J_{\bm\ell}} \frac{1}{2\ell_j}\(\sum_{m\in \Z} I_m^{(j)}I_{k-m-1}^{(p-j)}\)\ \text{if}\ k\ne 0,\\
  L^*_0&=& \sum_{j\in J_{\bm\ell}} \frac{1}{2\ell_j}\(\sum_{m<0} \elei mj\elei {-m-1}{p-j}+
  \sum_{m\geq 0}\elei {-m-1}{p-j}\elei mj\)+\sum_{j\in J_{\bm\ell}}\frac{j(p-j)}{4 p^2}.
 \end{eqnarray*}
 Furthermore, if both $V$ and $W$ are irreducible, then so is the $\g$-module $V\otimes W$.
\end{proposition}
\begin{proof}
 Let $\bm{\ell}$, $c$, $V$ and $W$ be as in the proposition and write
 \[
  L_0(z)=\sum_{m\in \Z} L_m z^{-m-2}\quad\text{and}\quad
  L_j(z)=\sum_{m\in \Z} I_m^{(i)}z^{-m-1}\quad (j=1,2,\dots,p-1).
 \]
 By Lemma \ref{LAiso} and Proposition \ref{chartwistmod},
 $W$ is a $\sigma_{J_{\bm{\ell}}^c}$-twisted $M_{\CN_{J_{\bm{\ell}}^c}}$-module such that
 \begin{equation} \label{eq1}
  Y_W(N_j(-1){\bm 1},z)=L_j(z)\quad\text{for } j\in J_{\bm\ell}^c,
 \end{equation}
 and
 \begin{equation} \label{eq2}
   Y_W(K_0(-1),z)=c-|J_{\bm \ell}|,\quad Y_W(K_i(-1),z)=0\quad\text{for }i\in (J_{\bm\ell}^c)^*.
 \end{equation}
 From \eqref{eq2}, it follows that $W$ is  a $\sigma_{J_{\bm{\ell}}^c}$-twisted $M_{\CN_{J_{\bm{\ell}}^c}}(c-|J_{\bm \ell}|,{\un 0})$-module.
 Similarly,  $V$ is a $\sigma_{J_{\bm{\ell}}}$-twisted $M_{\CN_{J_{\bm{\ell}}}}$-module such that
 \begin{equation} \label{eq3}
  Y_W(N_j(-1){\bm 1},z)=L_j(z) \quad\text{for }j\in J_{\bm\ell}
 \end{equation}
 and
 \begin{equation}\label{eq4}
  Y_W(K_i(-1),z)=\ell_i\quad\text{for }i\in J_{\bm\ell}.
 \end{equation}
 One concludes from \eqref{eq4} that $V$ is a $\sigma_{J_{\bm{\ell}}}$-twisted
 $M_{\CN_{J_{\bm{\ell}}}}({\bm \ell}^*)$-module.
 Furthermore, by \eqref{key} and \eqref{eq3}, we have that
 \begin{equation} \label{eq5}
  Y_V(\omega_{J_{\bm\ell}}({\bm\ell}^*),z)= L_{\bm\ell^*}(z)
  :=\sum_{j\in J_{\bm\ell}}\frac{1}{2\ell_j}\(:L_j(z) L_{p-j}(z):
  +\ell_j\frac{j(p-j)}{2p^2}z^{-2}\).
 \end{equation}

 Using Proposition \ref{vaiden} and Lemma \ref{autiden}, we find that
 $W\otimes V$ is a $\sigma$-twisted $M_{\CN}(c,{\bm \ell})$-module such that (see \eqref{conformaliden})
 \begin{equation}\begin{split}\label{eq6}
  Y_{W\otimes V}(N_0(-1){\bm 1},z)&=Y_W(N_0(-1){\bm 1},z)\otimes Y_V(\omega_{J_{\bm\ell}},z),\\
  Y_{W\otimes V}(N_i(-1){\bm 1},z)&=\begin{cases}
  Y_W(N_i(-1){\bm 1},z)\otimes 1\quad \text{if}\ i\notin J_{\bm \ell},\\
  1\otimes Y_V(N_i(-1){\bm 1},z)\quad \text{if}\ i\in J_{\bm \ell},
  \end{cases}
  \end{split}
 \end{equation}
 where $i=1,2,\dots,p-1$.
 Again by Proposition \ref{chartwistmod} with $J=\{0,1,2,\dots,p-1\}$, $W\otimes V$ is a restricted $\g$-module of central charge $c$ and of level ${\bm \ell}$  such that (see \eqref{eq1}, \eqref{eq3}, \eqref{eq5} and \eqref{eq6})
 \begin{equation*}\begin{split}
  L_0(z)=Y_{W\otimes V}(N_0(-1){\bm 1},z)&=Y_W(N_0(-1){\bm 1},z)\otimes Y_V(\omega_{J_{\bm\ell}},z)=L_0(z)\otimes L_{{\bm\ell}^*}(z),\\
  L_i(z)=Y_{W\otimes V}(N_i(-1){\bm 1},z)&=\begin{cases}
  Y_W(N_i(-1){\bm 1},z)\otimes 1=L_i(z)\otimes 1\quad \text{if}\ i\notin J_{\bm \ell},\\
  1\otimes Y_V(N_i(-1){\bm 1},z)=1\otimes L_i(z)\quad \text{if}\ i\in J_{\bm \ell},
  \end{cases}
  \end{split}
 \end{equation*}
 where $i=1,2,\dots,p-1$. This proves the first assertion and the second one is obvious.
\end{proof}

\bigskip

\textbf{Acknowledgments:}
C. Xu is supported by the National Natural Science Foundation of China (No. 12261077).
F. Chen is supported by the Natural Science Foundation of Xiamen, China (No. 3502Z202473005), the National Natural Science Foundation of China (No. 12471029) and the Fundamental Research Funds for the Central Universities (No. 20720230020).
S. Tan is supported by the National Natural Science Foundation of China (No. 12131018).

\end{document}